\documentclass{amsart}

\usepackage{geometry}
\usepackage{amssymb}
\usepackage{amsthm}
\usepackage{amsmath}
\usepackage{graphicx}
\usepackage{algorithm,algorithmic}
\usepackage[cmtip,all]{xy}

\newcommand{\texorpdfstring}[2]{#1}

\newcommand{\cA}{\mathcal{A}}

\newcommand{\cD}{\mathcal{D}}

\newcommand{\cH}{\mathcal{H}}

\newcommand{\cT}{\mathcal{T}}

\newcommand{\bA}{\mathbf{A}}

\newcommand{\bD}{\mathbf{D}}

\newcommand{\bbA}{\mathbb{A}}

\newcommand{\frakp}{\mathfrak{p}}

\DeclareMathOperator{\Hom}{Hom}

\DeclareMathOperator{\ord}{ord}

\DeclareMathOperator{\GL}{GL}

\DeclareMathOperator{\SL}{SL}

\DeclareMathOperator{\Res}{Res}

\DeclareMathOperator{\nrd}{nrd}

\DeclareMathOperator{\Stab}{Stab}
\DeclareMathOperator{\red}{red}

\DeclareMathOperator{\Dist}{Dist}

\DeclareMathOperator{\Up}{U_p}

\providecommand{\abs}[1]{\left\lvert#1\right\rvert}

\providecommand{\twomat}[4]{\left(\begin{array}{cc}#1&#2\\#3&#4\end{array}\right)}
\providecommand{\pseries}[2]{#1[\![ #2]\!]}
\newcommand{\padic}{$p$\text{-adic }}

\newcommand{\Qp}{\mathbf{Q}_p}
\newcommand{\Zp}{\mathbf{Z}_p}
\newcommand{\Cp}{\mathbf{C}_p}
\newcommand{\QQ}{\mathbf{Q}}
\newcommand{\FF}{\mathbf{F}}
\newcommand{\CC}{\mathbf{C}}

\newcommand{\ZZ}{\mathbf{Z}}
\newcommand{\PP}{\mathbf{P}}

\newcommand{\uhp}{\mathcal{H}}
\newcommand{\puhp}{{\mathcal{H}_p}}

\newcommand{\mtx}[4]{\left(\begin{matrix}#1&#2\\#3&#4\end{matrix}\right)}

\newcommand{\smtx}[4]{\left(\begin{smallmatrix}#1&#2\\#3&#4\end{smallmatrix}\right)}
\newcommand{\ol}[1]{{\overline{#1}}}

\DeclareMathOperator{\val}{val}

\providecommand{\smtx}[4]{\left(\begin{array}{cc}#1&#2\\#3&#4\end{array}\right)}

\newcommand{\ebar}{\overline{e}}

\providecommand{\demph}[1]{\emph{#1}}
\newcommand{\vphi}{\varphi}

\newcommand{\injects}{\hookrightarrow}
\newcommand{\surjects}{\twoheadrightarrow}
\newcommand{\Rmax}{{R^{\text{max}}}}

\theoremstyle{plain}
\newtheorem{thm}{Theorem}[section]
\newtheorem{lem}[thm]{Lemma}
\newtheorem{prop}[thm]{Proposition}

\theoremstyle{definition}
\newtheorem{dfn}[thm]{Definition}
\newtheorem{ex}[thm]{Example}
\newtheorem{rmk}[thm]{Remark}

\begin{document}
\title[Computing fundamental domains]{Computing fundamental domains for the Bruhat-Tits tree for $\GL_2(\Qp)$, $p$-adic automorphic forms, and the canonical embedding of Shimura curves.}

\author{Cameron Franc, Marc Masdeu}




\date{\today}

\begin{abstract}
We describe algorithms that allow the computation of fundamental domains in the Bruhat-Tits tree for the action of discrete groups arising from quaternion algebras. These algorithms are used to compute spaces of rigid modular forms of arbitrary even weight, and we explain how to evaluate such forms to high precision using overconvergent methods. Finally, these algorithms are applied to the calculation of conjectural equations for the canonical embedding of $p$-adically uniformizable rational Shimura curves. We conclude with an example in the case of a genus $4$ Shimura curve.
\end{abstract}
\maketitle

\section{Introduction}
\label{sec:intro}
Fix a prime $p$. This article describes algorithms related to quarternionic groups acting on the Bruhat-Tits tree $\cT$ for $\GL_2(\Qp)$. More precisely, fix a squarefree integer $N^-$ that is coprime to $p$, and that is divisible by an odd number of primes. Let $B/\QQ$ denote the definite quaternion algebra of discriminant $N^-$. Let $N^+$ be a positive integer coprime to $pN^-$. Let $R \subseteq B$ be an Eichler $\ZZ$-order of level $N^+$, and let $\Gamma$ denote the subgroup of elements of reduced norm $1$ in $R[1/p]$. This group acts acts on $\cT$ via a fixed splitting $ \iota \colon B\otimes_\QQ \Qp \stackrel{\sim}{\to} M_2(\Qp)$. The first algorithm that we describe yields, among other data, a fundamental domain for this action.
\begin{thm}
\label{t:mainthm}
There exists an explicit algorithm for computing the following data:
\begin{enumerate}
\item a finite connected subtree $\cD$ of $\cT$ whose edges comprise a complete set of distinct orbit representatives for the action of $\Gamma$ on the edges of $\cT$;
\item a list of the stabilizers in $\Gamma$ for all of the edges and vertices of $\cD$;
\item a pairing of the boundary vertices of $\cD$ that describes how the boundary vertices are identified in the quotient graph $\Gamma\backslash \cT$. 
\end{enumerate}
This algorithm runs in time
\[
O\left(\frac{(\log g)^3g(p^3+2g)}{(\log p)^3p^2}\right),
\]
where $g$ is the genus of the Shimura curve attached to $R$.
\end{thm}
In particular, for fixed $p$ our algorithm runs in time $O\left(g^2(\log g)^3\right)$. We expect that Theorem \ref{t:mainthm} will generalize to quaternion algebras over totally real fields.

While this paper provides a purely local method for computing a fundamental domain for the action of $\Gamma$ on the Bruhat-Tits tree, the reader should note that one could use strong approximation and global quaternion algorithms developed by Kirschmer and Voight in~\cite{kirschmer-voight} in order to obtain the same data. This is done, for example, in Sijsling's Ph.D. thesis (\cite[Chapter 5]{sijslingthesis}). Both the algorithm described in this paper and the algorithms of Kirschmer-Voight for computing ideal class sets in global quaternion orders hinge on short-vector searches in certain arithmetic lattices. Further, if one is only interested in obtaining the quotient graph $\Gamma \backslash \cT$, without any kind of boundary edge pairing, then this can be obtained efficiently from the theory of Brandt matrices, as explained in~\cite[Section~3]{MR2418792}. It is also worth noting that B{\"o}ckle and Butenuth~\cite{bocklebutenuth} have developed an algorithm for computing quaternionic fundamental domains in the Bruhat-Tits tree for $\GL_2(\pseries{\FF_q}{T})$.

This paper grew out of an attempt to extend the algorithms in Matthew Greenberg's thesis~\cite[Appendix]{greenberg2006thesis} to groups arising from quaternion algebras other than the rational Hamilton quaternions. Our interest in such algorithms stemmed from a desire to use the fundamental domain, and the work of Darmon-Pollack~\cite{pollackoverconvergent}, to compute values of rigid analytic modular forms and anti-cyclotomic $p$-adic $L$-functions. In the latter half of this article we extend the algorithm for computing values of rigid modular forms of weight $2$ explained in~\cite[Part I]{greenberg-thesis} to arbitrary even weight. This uses the overconvergent coefficient modules that were introduced in~\cite{pollackoverconvergent}.

As an application of the above algorithms we devise a method to compute equations for the canonical embedding of a $p$-adically uniformizable rational Shimura curve. A similar method was exploited by Kurihara in~\cite{MR1297415}, although the models previously found did not correspond to the canonical embedding, and he worked systematically with explicit bases of so-called $p$-adic Poincare series. Just as in Kurihara's work, at one step in our calculation we must recognize $p$-adic approximations to rational numbers, and this means that our equations are only conjectural. We illustrate the method in detail by computing conjectural defining equations for the genus $4$ curve $X_0(53\cdot 2,1)$. We uniformize at the prime $p = 53$ and show that our equations give an integral model that has semistable reduction at $53$. Note, though, that our method is a little ad-hod, and it seems unlikely that these sorts of computations can be developed into a generic algorithm for computing equations of Shimura curves.

This article is arranged as follows: in Section~\ref{sec:btt} we introduce notation and describe the Bruhat-Tits tree of $\GL_2(\QQ_p)$ as a retract of the $p$-adic upper half plane. In Section~\ref{sec:effcomp} we define the fundamental domains that are treated in this paper, and describe an algorithm for their computation. Section~\ref{sec:meas-etc} describes the spaces of harmonic cocycles and their relationship to modular forms.  In Section~\ref{sec:autforms} we define automorphic forms on $\GL_2(\QQ_p)$ with arbitrary coefficients, and extend the algorithm of~\cite{greenberg-thesis} to higher weight. In the final Section~\ref{sec:applications} we apply the previous algorithms to the evaluation of rigid modular forms, and give a method that takes advantage of this efficient evaluation to the computation of equations of Shimura curves. All the computations have been done using a~\emph{Sage} (see~\cite{sage}) implementation of the algorithms, and the code is available on the second author's website.

The authors wish to thank Gebhart B\"ockle and Ralph Butenuth for some helpful discussions about their work. The authors also wish to thank Henri Darmon, Matthew Greenberg, Jenna Rajchgot, Victor Rotger and John Voight for their generous advice and encouragement.

\section{The Bruhat-Tits tree}
\label{sec:btt}
This section introduces the Bruhat-Tits tree and explains its relation to the $p$-adic upper half plane.
\subsection{Definition}
\label{ssec:btdef}
Fix a rational prime $p$. The Bruhat-Tits tree $\cT$ for $\GL_2(\Qp)$ is the following graph: the vertices of $\cT$ are the homothety classes of
$\Zp$-lattices in $\Qp^2$, where $\Qp^2$ is regarded as a space of column 
vectors. Let $V(\cT)$ denote the vertex set of $\cT$. Two 
vertices are joined by an unordered edge if there exist representative lattices 
$\Lambda_1$ and $\Lambda_2$ for the respective vertices such that
\[
  p\Lambda_1 \subsetneq \Lambda_2 \subsetneq \Lambda_1.
\]
If $\Lambda$ is a lattice in $\Qp^2$ then  $[\Lambda]$ will denote the corresponding homothety class. The set of ordered pairs of adjacent vertices of $\cT$ wil be denoted $E(\cT)$ and elements of $E(\cT)$ will be refered to as ordered edges of $\cT$.
\begin{prop}
  The Bruhat-Tits tree for $\GL_2(\Qp)$ is a connected tree such that each vertex has degree $p+1$.
\end{prop}
\begin{proof}
See~\cite[Section 1.3.1, Proposition 8]{DT08}.
\end{proof}

The group $\GL_2(\Qp)$ acts on $\cT$: the action on vertices is given by matrix multiplication. This preserves adjacency and thus describes an action on $\cT$ by graph automorphisms. If $A \in \GL_2(\Qp)$ then $[A]$ will denote the homothety class of the lattice in $\Qp^2$ that is spanned by the columns of $A$. The map $A \mapsto [A]$ induces an equivariant bijection between $\GL_2(\Qp)/\Qp^\times\GL_2(\Zp)$ and $V(\cT)$ for the natural left actions of $\GL_2(\Qp)$.

Let $(v_0, v_1)$ denote the ordered edge of $\cT$ from the distinguished vertex $v_0 = [\Zp^2]$ to the vertex $ v_1 = \smtx 100p v_0$. The stabilizer of $v_1$ for the action of $\GL_2(\Qp)$ is precisely
\[
  \Qp^\times \cdot \smtx 100p \GL_2(\Zp)\smtx 100p^{-1}.
\]
Thus, if one writes
\begin{align*}
  \Gamma_0(p\ZZ_p)& = \GL_2(\Zp) \cap \left(\smtx 100p \GL_2(\Zp)\smtx 100p^{-1}\right) \\
&= \left\{\smtx abcd  \in \GL_2(\Zp) ~\Big{|}~ p | c \right\},
\end{align*}
then the stabilizer of $(v_0,v_1)$ for the action of $\GL_2(\Qp)$ is $\Qp^\times\cdot \Gamma_0(p\ZZ_p)$. The map $A \mapsto \left([A], \left[A\smtx 100p \right]\right)$ yields a bijection
\begin{equation}
\label{eq:ordedgeid}
\GL_2(\Qp)/\Qp^\times\cdot\Gamma_0(p\ZZ_p) \stackrel{\sim}{\to} E(\cT) 
\end{equation}
that is equivariant for the left action of $\GL_2(\Qp)$. The vertx $v_0$ will be refered to as the \emph{privileged vertex} and the ordered edge $(v_0, v_1)$ will be refered to as the \emph{privileged edge}. 

\begin{lem}
\label{lem:redlem}
There is a set of coset representatives $\{e_i\}_i$ for $\GL_2(\QQ_p)/\QQ_p^\times\cdot\Gamma_0(p\ZZ_p)$ given by matrices with coefficients in $\ZZ$. Moreover, there is an effective algorithm that, given any matrix in $g \in$ $\GL_2(\QQ_p)$, finds a corresponding scalar $\lambda\in\QQ_p^\times$ and matrix $t\in\Gamma_0(p\ZZ_p)$ satisfying $g\lambda t= e_i$. An analogous statement holds for $\GL_2(\Qp)/\Qp^\times\cdot\GL_2(\Zp)$.
\end{lem}
\begin{proof}
We will show that the matrices $e_i$ may be taken of the form:
\begin{align*}
\smtx{p^m}{0}{r}{p^n} &\quad 0\leq r < p^{n+1},&\smtx{0}{p^m}{p^n}{r}&\quad 0\leq r < p^n,
\end{align*}
for integers $m$, $n\geq 0$. Since we can scale by elements of $\QQ_p^\times$ we may assume without loss of generality that the matrix $g \in \GL_2(\QQ_p)$ belongs to $\GL_2(\ZZ_p)$ and that one of its entries has valuation $0$. Write $g = \smtx abcd$. Suppose that $\alpha = \val_p(a)\leq \beta = \val_p(b)$. Rescaling by an element in $\ZZ_p^\times$ allows us to assume that $c$ is a rational integer. Write $a = r p^\alpha$ for $r\in \ZZ_p^\times$, set $N=\val_p(ad-bc)-\alpha$, and let $s\in \ZZ$ satisfy $rs\equiv 1\pmod{p^{N+1}}$. Such an $s$ exists because $r$ is a $p$-adic unit. Define $c'$ as the integer in the range $\{1,\ldots,p^{N+1}\}$ such that $c'\equiv cs\pmod{p^{N+1}}$.
If we write
\[
g'=\smtx{p^{\alpha}}{0}{c'}{p^N},
\]
then one sees that $g'\Gamma_0(p\ZZ_p) = g\Gamma_0(p\ZZ_p)$.

Suppose now that $\alpha>\beta$. In that case we write $b = r p^\beta$ for $r\in\ZZ_p^\times$ and set $N=\val_p(ad-bc)-\beta$. Let $s\in\ZZ$ be an element such that $rs\equiv 1\pmod{p^N}$ and define $d'$ to be the integer in $\{1,\ldots,p^N\}$ such that $d'\equiv ds\pmod{p^N}$. If we write
\[
g'=\smtx{0}{p^{\beta}}{p^N}{d'},
\]
then one sees that $g'\Gamma_0(p\ZZ_p) = g\Gamma_0(p\ZZ_p)$.

A similar and slightly simpler proof applies to vertices.
\end{proof}
\begin{rmk}
\label{rmk:matrixreps}
The algorithm described in the proof of Lemma~\ref{lem:redlem} is used to encode the $\Gamma$-action on $\cT$ in terms of a fixed set of matrix representatives for the vertices and edges of $\cT$. The \emph{normalization} of a matrix in $\GL_2(\Qp)$ refers to the orbit representative of Lemma~\ref{lem:redlem} that represents the same vertex or edge as the given matrix. Note that since we will be working with a fixed $p$-adic precision, the complexity of normalization can be regarded as bounded by an absolute constant that depends only on the precision.
\end{rmk}

\subsection{Quaternionic action on  $\cT$}
\label{ssec:quataction}
Let $B/\QQ$ denote a definite quaternion algebra that is split at $p$. If $\ell$ is a place of $\QQ$, then we write $B_\ell = B \otimes_\QQ \QQ_\ell$. Thus $B_\infty$ is isomorphic with the Hamilton quaternions and $B_p \cong M_2(\Qp)$. Let $N^-$ denote the discriminant of $B$. Note that $N^-$ is coprime to $p$.

Let $\Rmax \subseteq B$ denote a maximal $\ZZ[1/p]$-order. Let $N^+$ denote a positive integer that is coprime to $pN^-$, and let $R \subseteq \Rmax$ denote an Eichler $\ZZ[1/p]$-order of level $N^+$. Since $B$ is unramified at $p$, it satisfies the Eichler condition for $\ZZ[1/p]$-orders and there hence exists a unique, up to conjugation by $B^\times$, Eichler $\ZZ[1/p]$-order of each level. Let $\iota$ denote a splitting isomorphism $\iota \colon B_p \cong M_2(\Qp)$ that satisfies $\iota(R_p^\text{max}) = M_2(\Zp)$. Let $\Gamma=\Gamma(p,N^-,N^+)$ denote the subgroup of elements of reduced norm $1$ in $R$. The group $\Gamma$ acts on $\cT$ via the splitting $\iota$.

A discrete subgroup of $\GL_2(\Qp)$ is said to be \emph{Schottky} if it acts without fixed points on the vertices of the Bruhat-Tits tree.
\begin{prop}
\label{p:quotproperties}
\begin{enumerate}
The following facts about $\Gamma$ are true.
\item The group $\Gamma $ is a finitely generated discrete subgroup of $\GL_2(\Qp)$ and the quotient $\Gamma \backslash \cT$ is finite.
\item There exists an integer constant $M \geq 1$ depending only on $pN^-$ such that if $N^+ \geq M$, the group $\Gamma$ is Schottky. In this case the abelianization of $\Gamma$ is a finite free $\ZZ$-module of rank $g = 1-V+E$, where $V$ and $E$ denote respectively the number of vertices and edges of $\Gamma\backslash \cT$.
\end{enumerate}
\end{prop}
\begin{proof}
See~\cite[Section~I.3]{gvdp80}.
\end{proof}

The main result of this paper is the description of an algorithm that outputs:
\begin{enumerate}
\item A connected subtree $\cD$ of $\cT$ which is a fundamental domain for the action of $\Gamma$ on $\cT$.
\item The map $\cT\to\Gamma\backslash\cT$ or, in other words, an efficient way to find, given $\sigma$ in $\cT$ a vertex or an edge, a corresponding vertex or edge $\ol\sigma$ in $\cD$ together with an element $\gamma\in\Gamma$ satisfying $\gamma \sigma = \ol \sigma$.\end{enumerate}

Our interest in the quaternionic group $\Gamma$ arises from the fact that $\Gamma \backslash \cT$ describes a bad special fiber of a Shimura curve, cf.~\cite{boutot-carayol}. In the Schottky case the corresponding special fiber has irreducible components in bijection with the vertices of $\Gamma \backslash \cT$. Each component is isomorphic with $\PP^1$ over $\FF_{p^2}$, and two components meet in an ordinary double point defined over $\FF_{p^2}$ if and only if the corresponding vertices of $\Gamma \backslash \cT$ are joined by an edge. Thus, Algorithm \ref{alg:fdomain} below computes the bad special fibers of integral models of Shimura curves over $\QQ$. In the final section we will describe how this algorithm can be applied to compute equations for integral models of Shimura curves.

\subsection{The $p$-adic upper half plane}
\label{sec:puhpdef}
Let $\puhp = \PP^1(\Cp) - \PP^1(\Qp)$ denote the \padic upper half plane, which is a Stein domain in the rigid analytic variety $\PP^1(\Cp)$, and thus inherits a structure of rigid analytic variety. The group $\GL_2(\Qp)$ acts on $\puhp$ via rigid analytic automorphisms via fractional linear transformation. For details consult \cite{DT08} or \cite[Chapter 5]{darmonbook}.

Let $\cA_0 \subseteq\puhp$ denote the following affinoid subset of $\PP^1(\Cp)$:
\[
  \cA_0 = \{z \in \PP^1(\Cp) ~|~ \abs{z} \leq 1, \abs{z-a} \geq 1 \text{ for } a = 0, \ldots, p-1\}. 
\]
The orbit of $\cA_0$ under the action of $\GL_2(\Qp)$ nearly tesselates $\puhp$: all that is missed is a collection of bounding annuli. This decomposition of $\puhp$ is closely related to the \emph{reduction map}: let $\cT_\QQ$ denote the topological realization of $\cT$ obtained by glueing rational intervals $[0,1] \cap \QQ$ according to the adjacency relations of $\cT$.
\begin{prop}
\label{p:reductionmap}
There is a continuous and surjective $\GL_2(\Qp)$-equivariant map
\[
  \red \colon \puhp \to \cT_\QQ
\]
that collapses affinoids $\gamma \cdot \cA_0$ for $\gamma \in \GL_2(\Qp)$ to vertices of $\cT_\QQ$, and which maps points in the bounding annuli to the edges of $\cT_\QQ$.
\end{prop}
For subgroups $\Gamma \subseteq \GL_2(\Qp)$, the quotient $\Gamma \backslash \uhp_p$ is compact if and only if $\Gamma \backslash \cT$ is a finite graph. The reduction map expresses the Bruhat-Tits tree as a skeleton of the $p$-adic upper half plane.

\section{Computing fundamental domains}
\label{sec:effcomp}
This section describes a computable criterion for determining whether two vertices or edges in the Bruhat-Tits tree are $\Gamma$-equivalent. This is then used to describe an algorithm that finds a fundamental domain for $\Gamma$ in $\cT$. 

\subsection{Solving the \texorpdfstring{$\Gamma$-equivalence}{equivalence} 
problem}
\label{ssec:equivprob}
Define $B/\QQ$, $N^-$, $N^+$, $R$, $\iota$ and $\Gamma$ as in Subsecton~\ref{ssec:quataction}. Write by $\nrd(x)$ the reduced norm of an element $x\in B$. Concretely, $\nrd$ is the composition $\det\circ \iota$. The group $\Gamma$ possesses an increasing filtration by finite sets which can be used to aid computations with the group. For each integer $n \geq 0$ set
\[
  \Gamma_n = \left\{\iota\left(\frac{x}{p^n}\right) ~\Bigg{|}~ x \in R \text{~and~} \nrd(x) = p^{2n}\right\}.
\]
Then $\Gamma_n \subseteq \Gamma_{n+1}$ since $x/p^n = (px)/p^{n+1}$ and 
$\Gamma = \bigcup_{n \geq 0} \Gamma_n$. Note that each $\Gamma_n$ is a finite set since $B$ is definite. 

Following~\cite{bocklebutenuth}, we introduce some generic notation for group actions. If $G$ is a group and $X$ is a left $G$-set, and if $u$ and $v$ are elements of $X$, then we write $\Hom_G(u,v)$ to denote the collection of all elements of $G$ that move $u$ to $v$. We also write $\Stab_G(u) = \Hom_G(u,u)$.

The central problem that arises when computing a fundamental domain for $\Gamma \backslash \cT$ is the computation of the sets $\Hom_\Gamma(u,v)$ for two vertices $u$ and $v$ of $\cT$ (or for two ordered edges). Our solution to this problem uses short vector searches in certain $\ZZ$-lattices of rank $4$. In the following exposition of our method we concentrate on the case of vertices. The case of ordered edges is treated analogously.

Assume that two vertices $u$ and $v$ of $\cT$ are represented by reduced matrices as in Lemma~\ref{lem:redlem}. We will abuse notation and denote the representing matrices also as $u$ and $v$. Define $m$ via the formula $2m=\val_p(\det u\det v)$. The integer $m$ is the half-length of the path that joins the vertex $u$ to the vertex $v$ passing through the privileged vertex $v_0$. Write $p^{a}=\det u$ and $p^{b}=\det v$, so that $2m=a+b$. Note that
\begin{align}
\label{eq:lattices}
  \Hom_\Gamma(u,v) &= \Hom_{\GL_2(\Qp)}(u,v) \cap \Gamma\\\notag
&= \left(v\Stab_{\GL_2(\Qp)}(v_0)u^{-1}\right)\cap \Gamma\\\notag
&= \left(v\Qp^\times\GL_2(\Zp)u^{-1} \right)\cap \Gamma.
\end{align}

\begin{lem}
If $m$ is not an integer then 
$\Hom_\Gamma(u,v) = \emptyset$. Otherwise one has
\[
\Hom_\Gamma(u,v) = p^{-m}vM_2(\ZZ_p)u^*\cap \Gamma,
\]
where $u^*$ is the matrix satisfying $uu^*=\det u$.
\end{lem}
\begin{proof}
Since $\Gamma \subseteq \SL_2(\Qp)$, the corollary to Proposition 1 of~\cite[Chapter 2, Subsection  1.2]{serretrees} shows that $m$ must be an integer for $u$ and $v$ to be equivalent under $\Gamma$, as the group $\Gamma$ preserves the parity of the distance between any two vertices. 

By Equation~\eqref{eq:lattices} it suffices to show that
\[
\left(v\Qp^\times\GL_2(\Zp)u^{-1} \right)\cap \Gamma = p^{-m}vM_2(\ZZ_p)u^*\cap \Gamma. 
\]
Write $z \in (v\QQ_p^\times\GL_2(\ZZ_p)u^{-1})\cap \Gamma$ as $z=$ $v \lambda g u^{-1}$ for $\lambda\in\QQ_p^\times$ and $g\in\GL_2(\ZZ_p)$. Observe that $z=$ $p^{-a}v\lambda g u^*$. Since $z \in \Gamma$ we have $\det z = 1$ and hence $\lambda^2 p^{b-a} \sigma =1$, where $\sigma=\det g\in\ZZ_p^\times$. Therefore $2\val_p(\lambda)=a-b$ and 
thus $\val_p(p^{-a}\lambda)=-m$. We conclude that 
$p^{-a}\lambda g$ belongs to $p^{-m} \GL_2(\ZZ_p)\subseteq p^{-m} M_2(\ZZ_p)$, and thus
\[
\left(v\Qp^\times\GL_2(\Zp)u^{-1} \right)\cap \Gamma \subseteq p^{-m}vM_2(\ZZ_p)u^*\cap \Gamma. 
\]

Conversely if one writes $z \in p^{-m}vM_2(\Zp)u^*\cap \Gamma$ as $z = p^{-m}vgu^*=$ $p^{a-m}vgu^{-1}$, for $g \in M_2(\Zp)$, then one checks that $\det g = 1$ and hence $g \in \GL_2(\Zp)$. 
\end{proof}

The group $\Gamma$ is defined via the splitting $\iota$, which can be described as a matrix with $p$-adic entries. Unfortunately this matrix can only be stored on a computer up to a finite precision. Hence, when using $p$-adic methods, the image of $\iota$ can not be computed exactly. This complicates the problem of determining if a given matrix in $M_2(\Qp)$ belongs to $\Gamma$. In order to deal with this difficulty we introduce the notion of an approximation to a $\Qp$-linear map.
\begin{dfn}
  Let $n \geq 0$ be an integer, let $V$ and $W$ be two finite dimensional $\Qp$-vector spaces, and let $\Lambda_V \subseteq V$ and $\Lambda_W\subseteq W$ be $\Zp$-lattices. Let $f \colon V \to W$ be a $\Qp$-linear map satisfying $f(\Lambda_V) \subseteq \Lambda_W$. Then an \emph{approximation} of $f$ to precision $n$ is a $\Qp$-linear map $g \colon V \to W$ such that $g \equiv f \pmod{p^n}$ when restricted to $\Lambda_V$.
\end{dfn}

\begin{lem}
\label{lem:vertexlemma}
Let $u$ and $v$ be matrices in $M_2(\ZZ)\cap \GL_2(\QQ_p)$ representing two vertices of $\cT$. Let $f \colon M_2(\Qp) \to B_p$ be an approximation of $\iota^{-1}$ to $p$-adic precision $2m$ relative to the orders $M_2(\Zp)$ and $R_p$. Define a $\ZZ$-lattice in $B$ as follows:
\[
  \Lambda(u,v) = f(vM_2(\ZZ_p)u^*)\cap R + p^{2m+1}R.
\]
Then $\Hom_\Gamma(u,v)$ is nonempty if and only if the shortest vectors in $\Lambda(u,v)$ have reduced norm $p^{2m}$.
\end{lem}
\begin{proof}
  First note that, since $\iota$ transforms the reduced norm of $B_p$ into the determinant of $M_2(\Qp)$, elements in $\Lambda(u,v)$ have reduced norm of valuation at least $2m$. Let $g\in \Hom_\Gamma(u,v)$ and write $g=p^{n-m}v x u^*=\iota(y)$, where $x\in M_2(\ZZ_p)$ and $y\in p^n\Gamma_n$. The element $\lambda=p^{m-n}f(g)$ is a shortest vector in $\Lambda(u,v)$ of reduced norm $p^{2m}$: since it clearly has reduced norm $p^{2m}$ and elements of $\Lambda(u,v)$ have reduced norm at least $p^{2m}$, it suffices to show that $\lambda \in \Lambda(u,v)$.

Consider $\lambda'=\iota^{-1}(p^{m-n}g)$ and note that $\lambda'$ belongs to $R$: clearly $\lambda'=p^{m-n}y$ belongs to $R[1/p]$. But also $\lambda'=\iota^{-1}(vxu^*)$ is $p$-integral, since $vxu^*$ is so, and $\iota$ preserves $p$-integrality. Therefore $\lambda'\in R[1/p]\cap R^\text{max}_p=R$. Thus, by definition of $f$, we see that $\lambda'$ and $\lambda$ are congruent modulo $p^{2m+1}$ and we can write $\lambda=\lambda'+p^{2m+1}\alpha$ for $\alpha\in R$. This shows that $\lambda'$ belongs to $\Lambda(u,v)$, and therefore so does $\lambda$.

Conversely suppose that $\lambda\in \Lambda(u,v)$ is of reduced norm $p^{2m}$. Write $\lambda$ as $\lambda=f(v x u^*)+p^{2m+1}\alpha$ for $x\in \GL_2(\ZZ_p)$ and $\alpha\in R$. Again, by definition of $f$, we can rewrite this expression as $\lambda=\iota^{-1}(v x u^*)+p^{2m+1}\alpha'$ for $\alpha'\in R$. We claim that $\iota(\lambda)$ belongs to $\Hom_\Gamma(u,v)$. First, note that $\iota(\lambda/p^m)$ is indeed an element of $\Gamma$. It remains to show that $\iota(\lambda)$ takes the vertex corresponding to $u$ to that corresponding to $v$. For this we write
\[
\iota(\lambda)=vxu^*+p^{2m+1}\iota(\alpha')=v(x+p^{2m+1}v^{-1}\iota(\alpha')(\det u)^{-1}u)u^*,
\]
and note that the matrix
\[
x+p^{2m+1}v^{-1}\iota(\alpha')(\det u)^{-1}u=x+p (p^bv^{-1})\iota(\alpha')u
\]
belongs to $\GL_2(\ZZ_p)$ because $x$ does. Therefore it stabilizes the vertex $v_0$, and this concludes the proof.
\end{proof}

\begin{rmk}
If $u$ and $v$ are ordered edges of $\cT$, then the exact analogue of Lemma~\ref{lem:vertexlemma} is true with the $\ZZ$-lattice $\Lambda(u,v)$ replaced with
\[
\Lambda'(u,v)=f(v \Gamma_0(p\ZZ_p) u^*)\cap R +p^{2m+1}R.
\]
\end{rmk}

\begin{rmk}
Since the lattice $\Lambda(u,v)$ contains $p^{2m+1}R$, Lemma~\ref{lem:vertexlemma} allows $\Lambda(u,v)$ to be described on a computer using $p$-adic approximations for a basis of $R$ as long as the approximations have at least $2m$ digits of $p$-adic precision. One may then use standard techniques, like the LLL algorithm as explained in~\cite[Section~2.6]{cohen1993course}, to find the shortest vectors in $\Lambda(u,v)$. These algorithms have a complexity of $O(m^3)$. This yields an efficient algorithm for determining whether vertices or edges of $\cT$ are equivalent under $\Gamma$.
\end{rmk}

\subsection{Computing fundamental domains in the tree}
\label{ssec:fdomain}
In this subsection edge means \emph{ordered edge}. If $E$ is a set of edges of $\cT$ and if $v$ is a vertex, we denote by $E(v)$ the set edges of $E$ that originate at $v$.

Lemma~\ref{lem:vertexlemma} allows one to compute the quotient $\Gamma \backslash \cT$ in a straightforward way. Algorithm~\ref{alg:fdomain} which we describe below in fact computes a fundamental domain for the action of $\Gamma$ on $\cT$, and this data is richer than the data of the quotient graph. In analogy with the case of Riemann surfaces uniformized by $\uhp$, a \emph{fundamental domain} in $\cT$ for the action of $\Gamma$ consists of:
\begin{enumerate}
\item a connected subtree $\cD \subseteq \cT$ whose edges form a set of distinct coset representatives for the action of $\Gamma$ on the edges of $\cT$;
\item all nontrivial edge and vertex stabilizers (in the non-Schottky case);
\item a collection of triples $(u,v,\gamma)$ where $u$ and $v$ are distinct boundary vertices of $\cD$, and $\gamma \in \Gamma$ satisfies $\gamma u = v$. One triple is computed for each pair of identified boundary vertices. The data of all such triples is refered to as a \emph{boundary pairing}.
\end{enumerate}
\begin{rmk}
  Let $\cD \subseteq \cT$ denote a connected subtree whose edges form a complete set of distinct coset representatives for the action of $\Gamma \backslash E(\cT)$. If $u$ is a boundary vertex of $\cD$, then it must be $\Gamma$-equivalent to at least one other boundary vertex: if were not, then the edges of $\cT$ adjacent to $u$ and outside $\cD$ could not be $\Gamma$-equivalent to any edges in $\cD$, which contradicts the fact that the edges of $\cD$ contain a full set of coset representatives. One similarly uses the distinctness of the coset representatives in $\cD$ to show that $u$ can be so paired with a \emph{unique} other boundary vertex $v$.
\end{rmk}

\begin{algorithm}
\caption{Compute a fundamental domain for $\Gamma$ acting on $\cT$}
\label{alg:fdomain}
\begin{algorithmic}
\REQUIRE A prime $p$, an order $R \subseteq B$ as above, and a splitting $\iota_p\colon R_p\cong M_2(\ZZ_p)$.
\ENSURE A fundamental domain together with an edge pairing.
\STATE Select a vertex $v_0\in V(\cT)$ to begin.
\STATE Initialize a queue $W$ with $v_0$.
\STATE Initialize $E$ and $P$ as empty lists.
\WHILE{$W\neq \emptyset$}
 \STATE Pop $v$ from $W$.
 \FOR{$e\in E(\cT)(v)$}
  \IF{there is no $e'\in E(v)$ which is $\Gamma$-equivalent to $e$}
   \STATE Append $e$ to $E$.
   \IF{there is a vertex $v'\in W$ which is $\Gamma$-equivalent to $t(e)$}
   \STATE Append $(t(e),v',\gamma)$ to $P$, where $\gamma\in\Gamma$ satisfies $\gamma t(e)=v'$.
   \ELSE
    \STATE Push $t(e)$ onto $W$.
   \ENDIF
  \ENDIF
 \ENDFOR
\ENDWHILE
\RETURN $E$, $P$.
\end{algorithmic}
\end{algorithm}

\begin{rmk}
The computation of $\Gamma$-invariant harmonic cocycles on $\cT$, as in Section \ref{sec:meas-etc} below, can be facilitated by storing extra data during the execution of Algorithm~\ref{alg:fdomain}. More precisely, to compute the ``harmonicity relations'', one needs the data of how all of the edges leaving a given vertex get identified under $\Gamma$. Algorithm~\ref{alg:fdomain} computes this data when testing for membership in the fundamental domain.
\end{rmk}

One sees that Algorithm~\ref{alg:fdomain} terminates, as the quotient $\Gamma\backslash\cT$ is a connected and finite graph. We wish to analyze the complexity of this algorithm in terms of the genus $g$ of the corresponding Shimura curve. Let $X_0(pN^-,N^+)$ denote the Shimura curve associated to the Eichler order of level $N^+$ in the \emph{indefinite} quaternion algebra with discriminant $pN^-$. The following genus formula is due to Ogg.
\begin{thm}
\label{thm:genusformula}
The genus $g$ of $X_0(pN^-,N^+)$ satisfies
\[
g=1+\frac{pN^-N^+}{12}\prod_{\ell\mid pN^-}\left(1-\frac 1\ell\right)\prod_{\ell\mid N^+}\left(1+\frac 1 \ell\right)-\frac{e_3}{3}-\frac{e_4}{4},
\]
where the $e_k$ are defined as
\begin{align*}
e_k &=\prod_{\ell\mid pN^-}\left(1-\left(\frac{-k}{\ell}\right)\right)\prod_{\ell \| N^+}\left(1+\left(\frac{-k}{\ell}\right)\right)\prod_{\ell^2\mid N^+} \nu_\ell(k),\\
 \nu_\ell(k) &=\begin{cases}
2&\text{ if } \left(\frac{-k}{\ell}\right)=1,\\
0&\text{ else.}
\end{cases}
\end{align*}
Here $\left(\frac{\cdot}{\cdot}\right)$ stands for Kronecker's quadratic symbol.
\end{thm}
\begin{proof}
See~\cite[pg. 280,301]{MR717598}.
\end{proof}

In the following result we assume that $\Gamma$ is Schottky, but see Remark~\ref{rmk:nonschottky} for how the result can be used in the non-Schottky case.
\begin{prop}
Suppose that $\Gamma$ is Schottky. Let $n=n(p,N^-,N^+)$ be the number of sets of the form $\Hom_\Gamma(u,v)$ that need to be computed in order to find a fundamental domain for $\Gamma \backslash \cT$. Then
\[
n\leq \frac{(g-1)(p^3-2p+2g-1)}{(p-1)^2}.
\]
\end{prop}
\begin{proof}
   Let $V$ (resp. $E$) denote the number of vertices (resp. edges) in a fundamental domain. The algorithm terminates after doing at most $pE$ comparisons among edges, and at most ${V(V-1)/2}$ comparisons among vertices. Therefore $n=pE+V(V-1)/2$. Since we are assuming that $\Gamma$ is Schottky, the quotient $\Gamma\backslash\cT$ is $(p+1)$-regular, and $E$ and $V$ are related by the formula $2E=(p+1)V$. Hence
\[
V=\frac{2(g-1)}{p-1},\quad E=\frac{(p+1)(g-1)}{p-1}.
\]
The lemma follows.
\end{proof}

\begin{rmk}
\label{rmk:nonschottky}
  Since $\Gamma$ is assumed to be Schottky, $\Gamma\backslash \cT$ is Ramanujan. This implies, according to~\cite[Prop~7.3.11]{MR2569682}, that the diameter of $\Gamma\backslash\cT$ is $O(\log  g/\log p)$. Therefore the running time of Algorithm~\ref{alg:fdomain} is
\[
O\left(\frac{(\log g)^3(g-1)(p^3-2p+2g-1)}{(\log p)^3 (p-1)^2}\right)=O\left(\frac{(\log g)^3g(p^3+2g)}{(\log p)^3p^2}\right).
\]
For a general $\Gamma$, Proposition~\ref{p:quotproperties} provides a finite index normal Schottky subgroup $\Gamma'\subset\Gamma$ and one can use the Schottky case to deduce that the running time for $\Gamma$ is
\[
O\left(\frac{(\log Mg)^3(Mg-1)(p^3-2p+2Mg-1)}{(\log p)^3 (p-1)^2}\right)=O\left(\frac{(\log Mg)^3Mg(p^3+2Mg)}{(\log p)^3p^2}\right),
\]
where $M$ is a positive integer depending only on $pN^-$, as in Proposition~\ref{p:quotproperties}.
\end{rmk}

Algorithm \ref{alg:fdomain} can be improved in the case of a Schottky group. If a vertex $v$ has a trivial stabilizer in $\Gamma$, then all the edges adjacent to $v$ are necessarily inequivalent under $\Gamma$, and so one need not test for membership in a fundamental domain: if one accepts the vertex in the domain, then one must also add all of the adjacent edges to the domain. This allows one to avoid making \emph{any} edge comparisons in Algorithm~\ref{alg:fdomain} in the Schottky case. The running time then improves to
\[
O\left(\frac{(\log g)^3g^2}{(\log p)^3p^2}\right).
\]
In terms of the quotient graph, this improvement takes into account the fact that if $\Gamma$ is Schottky, then $\Gamma\backslash \cT$ is a $(p+1)$-regular connected multigraph without loops.

\subsection{Examples}
\label{sec:examples}
First take $p = 2$, $N^- = 13$ and $N^+=1$. The output of Algorithm \ref{alg:fdomain} is pictured in Figure~\ref{graph1_fundom} and the corresponding quotient graph is pictured in Figure~\ref{graph1}. This corresponds to the special fiber at $2$ of the Shimura curve  $X_0(26,1)$. The genus of this curve is $2$.
\begin{figure}[h]
  \centering
\includegraphics[width=1in]{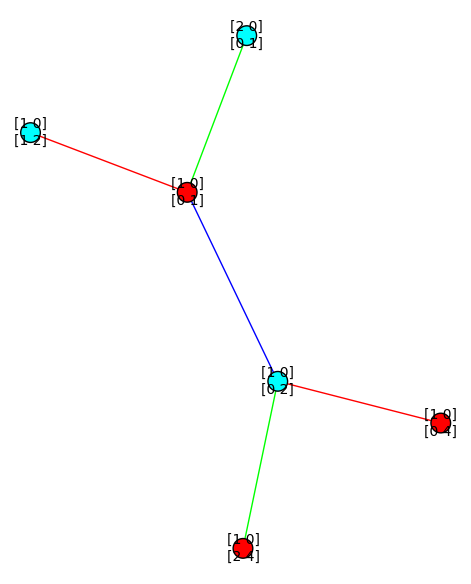}
  \caption{Fundamental domain for $\Gamma(2,13,1)$}
  \label{graph1_fundom}
\end{figure}

\begin{figure}[h]
  \centering
\includegraphics[width=1in]{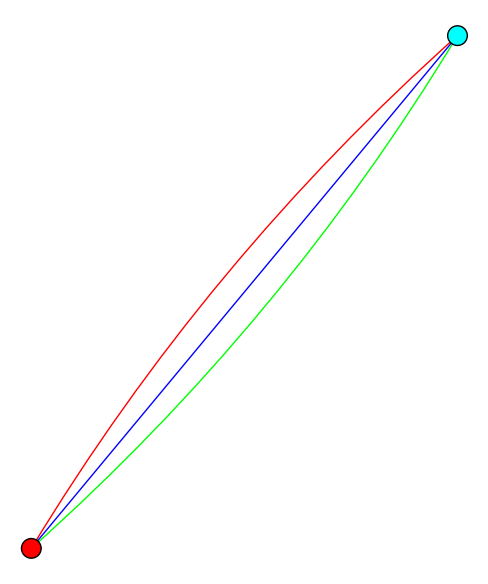}
  \caption{$\Gamma(2,13,1)\backslash \cT$}
  \label{graph1}
\end{figure}
In our next example we increase the level $N^+$ and consider the quotient map
\[\Gamma(2,13,3)\backslash \cT \to \Gamma(2,13,1)\backslash\cT.\]
This corresponds to the covering map $X_0(26,3) \to X_0(26,1)$. The colors in Figures \ref{graph2_fundom} and \ref{graph2} encode the identifications made by the quotient map.
\begin{figure}[h]
  \centering
\includegraphics[width=2in]{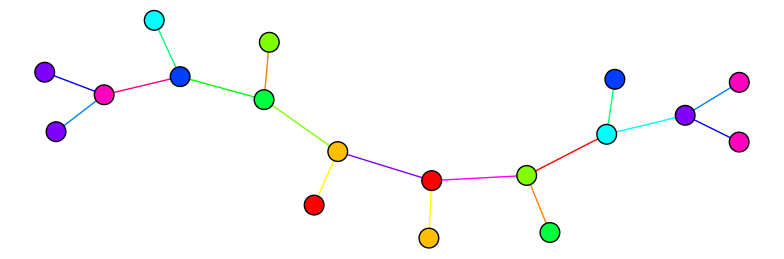}
  \caption{Fundamental domain for $\Gamma(2,13,3)$}
  \label{graph2_fundom}
\end{figure}

\begin{figure}[h]
  \centering
\includegraphics[width=2in]{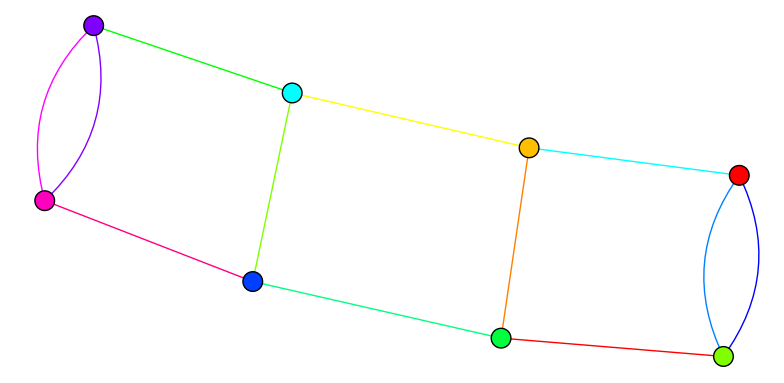}
  \caption{$\Gamma(2,13,3)\backslash\cT$}
  \label{graph2}
\end{figure}

Finally, by way of comparison, the quotient $\Gamma(211,1511,1)\backslash\cT$ is a graph with $254$ vertices and $26678$ edges and corresponds to a curve of genus $26425$. It took less than $3$ hours to compute on the same \emph{Intel Core i5}.

\section{Cocycles, distributions and modular forms}
\label{sec:meas-etc}
In this section we recall various definitions and results concerning rigid analytic modular forms on uniformized Shimura curves. Throughout this section $V$ denotes a $\Qp$-vector space endowed with a left action of $\GL_2(\Qp)$.

\subsection{Functions on the Bruhat-Tits tree}
\label{ssec:edgefuncs}
Let $C(\cT,V)$ denote the collection of $V$-valued functions on the ordered edges of the Bruhat-Tits tree. This space inherits a left action of $\GL_2(\Qp)$ defined for $c \in C(\cT,V)$ by the rule
\[
  (g\cdot c)(e) = g\cdot (c(g^{-1}\cdot e))
\]
for all $g \in \GL_2(\Qp)$ and all ordered edges $e$ of $\cT$. Recall that if $e=(v_1,v_2)$ is an ordered edge of $\cT$, we write $o(e)=v_1$ for the origin of $e$ and $t(e)=v_2$ for its terminus. 

Recall that a $V$-valued \demph{harmonic cocycle} on $\cT$ is a function $c \in C(\cT,V)$ such that $c(\ebar) = -c(e)$ for all edges $e$, and such that
\[
  \sum_{o(e) = v} c(e) = 0
\]
for all vertices $v$ of $\cT$. The space of harmonic cocycles is denoted by $C_h(\cT,V)$. Note that $\GL_2(\QQ_p)$ acts on $C_h(\cT,V)$. Given $\Gamma \subseteq \GL_2(\Qp)$ a subgroup, denote with
  $C(\Gamma,V) = C(\cT,V)^\Gamma$ and $C_h(\Gamma, V) = C_h(\cT,V)^\Gamma$  the invariant subspaces.

\subsection{Boundary distributions}
\label{ssec:meas}
\label{ssec:integrals}

In order to identify the computable space of harmonic cocycles with an arithmetically interesting space of modular forms, it is useful to introduce an intermediate space of distributions. 
\begin{dfn}
  A $V$-valued \demph{boundary distribution} on $\uhp_p$ is a $V$-valued distribution on $\PP^1(\Qp)$, that is, a function
\[
  \mu \colon \left\{\text{compact open sets in } \PP^1(\Qp)\right\} \to V
\]
that is finitely additive on disjoint unions.
\end{dfn}

We write $\Dist(V)$ for the $\Qp$-space of $V$-valued boundary distributions, with the left action of $\GL_2(\Qp)$ given by: $ (\gamma \cdot \mu)(U) = \gamma \cdot(\mu(\gamma^{-1}\cdot U))$; here $\gamma^{-1}\cdot U$ denotes the compact open obtained from $U$ by fractional linear transformation.

Let $\Dist_0(V)$ denote the subrepresentation of $\Dist(V)$ consisting of distributions $\mu$ satisfying $\mu(\PP^1(\Qp)) = 0$.


A compact open subset of $\PP^1(\Qp)$ can be written as a finite disjoint union of compact open balls. Hence, to specify a boundary distribution, it suffices to describe its values on compact open balls. The compact open balls of $\PP^1(\Qp)$ are in bijective correspondence with the ordered edges $e$ of the Bruhat-Tits tree. An \emph{end} of $\cT$ is a non-backtracking sequence $\{e_1,e_2,e_3,\ldots\}$ of edges in $E(\cT)$ satisfying $t(e_i)=o(e_{i+1})$, and two such ends are equivalent if the corresponding sequences differ only by a shift and a finite number of initial terms. The set of ends is naturally identified with $\PP^1(\Qp)$, and one can attach to an edge $e$  the open subset $U_e$ of $\PP^1(\Qp)$ consisting of all of the ends that pass through $e$. For more details the reader is invited to consult~\cite{DT08}.

 Note that for every ordered edge $e$ of the tree, and for every vertex $v$, one has:
\begin{align}
\label{eq:compactopens}
U_e \coprod U_{\ebar} &= \PP^1(\Qp),&\coprod_{o(e) = v} U_e &= \PP^1(\Qp).
\end{align}

This allows to associate to any given a boundary distribution $\mu$ satisfying $\mu(\PP^1(\Qp)=0$ a $V$-valued harmonic cocycle $c_\mu$, by the rule $c_\mu(e) = \mu(U_e)$. This induces a $\GL_2(\Qp)$-equivariant isomorphism
\[
  \Dist_0(V) \cong C_h(\cT,V),\quad \mu\mapsto c_\mu.
\]

If $\Gamma \subseteq \GL_2(\Qp)$ is a subgroup, then we let $\Dist_0(\Gamma,V)$ denote the $\Gamma$-invariant distributions in $\Dist_0(V)$. Since the isomorphism $\mu \mapsto c_\mu$ above is $\GL_2(\Qp)$-equivariant, it yields an isomorphism $\Dist_0(\Gamma,V) \cong C_h(\Gamma,V)$.

Locally constant functions with compact support on $\PP^1(\Qp)$ can be integrated against boundary distributions $\mu$ on $\uhp_p$. If $\mu$ is furthermore $\Gamma$-invariant for some cocompact subgroup $\Gamma$ of $\GL_2(\QQ_p)$, then $\mu$ can frequently be integrated against a wider class of functions. For example, if $\mu$ is $\Qp$-valued, then the cocompactness of $\Gamma$ implies that $\mu$ is bounded. Such distributions are called \emph{measures}, and they can be used to integrate any continuous function on $\PP^1(\Qp)$. We begin our discussion of integration by introducing the coefficient modules used to describe rigid analytic modular forms of arbitrary even weight.

For each integer $n \geq 0$ let $P_n \subseteq \Qp[X]$ denote the vector subspace of polynomials of degree at most $n$. We regard this as a right $\GL_2(\Qp)$-module via the action
\[
  P(X)\cdot \gamma = \det(\gamma)^{-n}(cX+d)^nP\left(\frac{aX+b}{cX+d}\right),\quad\quad \gamma = \mtx abcd \in \GL_2(\Qp).
\]
Let $V_n$ denote the linear dual of $P_n$. It inherits a left action of $\GL_2(\Qp)$ via the formula $(\gamma\cdot \omega)(P(X)) = \omega(P(X)\cdot \gamma)$.

Let $A_n$ denote the set of $\Qp$-valued functions on $\PP^1(\Qp)$ which are locally analytic except possibly for a pole at $\infty$ of order at most $n$. Functions in $A_n$ can be integrated against distributions $\mu \in \Dist_0(\Gamma, V_n)$. For a more precise statement of the following  theorem, and for a proof, consult \cite[Proposition 9]{padicpoissonkernel}, or the exposition in~\cite[Chapter~2, Theorem~2.3.2]{baker2008p}.
\begin{thm}[Amice-Velu, Vishik, Teitelbaum]
\label{thm:amicevelu}
There is a unique way to continuously extend a distribution $\mu \in \Dist_0(\Gamma, V_n)$ to a $\Gamma$-invariant distribution (written $\mu$ also) on $A_n$, satisfying furthermore:
\[
\int_{U_{a,n}}\sum_{m=0}^\infty c_m (x-a)^m d\mu = \sum_{m=0}^\infty c_m \int_{U_{a,n}} (x-a)^md\mu,\quad U_{a,n}=a+p^n\ZZ_p.
\]
\end{thm}

\subsection{Rigid analytic modular forms}
\label{ssec:ramfs}
Let $\Gamma \subseteq \SL_2(\Qp)$ denote a discrete and cocompact subgroup, so that $\Gamma \backslash \uhp_p$ is a rigid analytic curve. Assume that $\Gamma$ acts freely on $\cH_p$, which can always be achieved by passing to a normal subgroup of finite index. In this section $n\geq 0$ denotes an \emph{even} integer.

A \demph{rigid analytic modular form} for $\Gamma$ of weight $n+2$ is a rigid analytic function $f \colon \uhp_p \to \Cp$ such that
\[
  f\left(\frac{a\tau + b}{c\tau + d}\right) = (c\tau+d)^{n+2}f(\tau) \quad \text{ for all } \twomat abcd \in \Gamma.
\]
The $\Cp$-vector space of rigid analytic modular forms of weigh $n+2$ for $\Gamma$ is denoted by $S_{n+2}(\Gamma)$.

Suppose that $\mu \in \Dist_0(\Gamma, V_n)$. Then Theorem~\ref{thm:amicevelu} implies that we may define a function $f_\mu$ on $\uhp_p$ by setting
\[
  f_\mu(\tau) = \int_{\PP^1(\Qp)} \frac{d\mu(z)}{(\tau - z)}.
\]
This is in fact a rigid modular form for $\Gamma$ of weight $n+2$. The following is~\cite[Theorem 3]{padicpoissonkernel}.
\begin{thm}[Manin, Schneider, Teitelbaum]
  For each even integer $n\geq 0$, the association $\mu \mapsto f_\mu$ described above defines an isomorphism
\[
  \Dist_0(\Gamma, V_{n})\otimes_{\Qp} \Cp \cong S_{n+2}(\Gamma).
\]
The inverse is defined by $f \mapsto \mu_f$ where
\[
  \mu_f(U_e)(P(X)) = \Res_e(P(\tau)f(\tau)d\tau).
\]
Here $U_e$ denotes the compact open ball in $\PP^1(\Qp)$ consisting of ends containing the ordered edge $e$, and $\Res_e$ denotes the residue on the oriented annulus in $\uhp_p$ corresponding to the ordered edge $e$.
\end{thm}

\subsection{Computing the space of modular forms}
\label{sec:computeforms}
In this subsection we take $\Gamma \subseteq \GL_2(\Qp)$ to be a quaternionic group as defined in Subsection \ref{ssec:quataction}. Thus, $\Gamma$ acts discretely and cocompactly on $\cT$, and the stabilizer in $\Gamma$ of any edge or vertex of $\cT$ is finite. See for example~\cite[Chapitre V,\S 3]{vigneras} for more details.

The space of rigid analytic modular forms for $\Gamma$ can be computed in practice by computing the space $C_h(\Gamma, V_{n})$ and by using the identifications
\[
  C_h(\Gamma,V_{n}) \cong \Dist_0(\Gamma, V_{n}) \cong S_{n+2}(\Gamma).
\]
The computation of $C_h(\Gamma, V_{n})$ amounts to an exercise in linear algebra over $\Qp$. One challenge in this computation is to ensure that one has enough $p$-adic precision when storing elements of $V_{n}$ in order to compute the space of harmonic cocycles to the desired precision. One way to guarantee correct output is to use a dimension formula for these spaces.
\begin{thm}
 For each $k\geq 2$ and $N\geq 1$ denote by $\delta_k(N)$ the dimension of the space of (classical) cusp forms of weight $k$ and level $\Gamma_0(N)$. Set $d_k(1,N^+)=\delta_k(N^+)$ and define recursively $d_k(L,N^+)$ as
\[
d_k(L,N^+)=\delta_k(LN^+)-\displaystyle\sum_{\substack{m\mid L\\m\neq L}} \sigma(L/m) \, d_k(m,N^+),\quad L>1.
\]
where $\sigma(x)$ gives the number of divisors of the integer $x$. Then the dimension of the space $C_h(\Gamma,V_n)$, where $\Gamma=\Gamma(p,N^-,N^+)$ is as defined in Subsection~\ref{ssec:quataction} is $d_{n+2}(pN^-,N^+)$.
\end{thm}
\begin{proof}
  An easy exercise after noting that the Jacquet-Langlands correspondence implies
\[
\dim_\QQ C_h(\Gamma,V_n) = \dim_\QQ S_{n+2}\left(\Gamma_0(pN^-N^+)\right)^{pN^-\text{-new}}.
\]
Note also that for $k=2$ this gives a recursive formula for the genus $g$ of the quotient graph $\Gamma\backslash\cT$.
\end{proof}

\begin{rmk}
An alternative way to compute the space of harmonic cocycles that avoids issues of $p$-adic precision involves fixing a splitting of the quaternion algebra $B$ by a suitably chosen quadratic field $F$ that can be embedded in $\QQ_p$. One can then work with the analogue of $V_{n}$ defined over that quadratic field. Once an exact basis is found, it can be mapped into a $\QQ_p$-basis by the embedding $F\injects \QQ_p$.
\end{rmk}

\section{$p$-adic automorphic forms}
\label{sec:autforms}
In this section we provide some definitions and results that allow one to efficiently evaluate modular forms. These so-called overconvergent methods originated from work of Pollack and Stevens and were adapted to the quaternionic setting by Matthew Greenberg in~\cite{greenberg2006thesis}.

Throughout this section we let $B$ denote a definite quaternion algebra split at $p$ as in Subsection \ref{ssec:quataction}. The objects $N^-$, $N^+$, $R$ and $\Gamma$ are also defined as in Subsection \ref{ssec:quataction}.
\subsection{Coefficient modules}
\label{ssec:coeffs}
Let $\Sigma_0(p) \subseteq B_p^\times$ denote the following monoid
\[
  \Sigma_0(p) = \left\{ \twomat abcd \in M_2(\Zp) ~\Bigg{|}~ p \mid c,~ d \in \Zp^\times,~ ad - bc \neq 0\right\}. 
\]

\begin{dfn}
A \demph{coefficient module} is a $\QQ_p$-vector space endowed with a \emph{right} action of $\Sigma_0(p)$. A coefficient module $M$ is said to be \demph{pure} of weight $n$ if for all $x\in M$ and all $\lambda\in\ZZ_p^\times$,
\[
x\cdot \smtx \lambda 00\lambda =\lambda^n x.
\]
\end{dfn}

The following families of coefficient modules are the only ones that we require.
\begin{ex}
\label{ex:Vncoeffmodule}
For $n \geq 0$ let $P_n$ and $V_n = P_n^\vee$ be defined as in Subsection \ref{ssec:integrals}. 
The space $P_n$ of polynomials of degree $\leq n$ is endowed with a right action of $\GL_2(\Qp)$, and hence in particular of $\Sigma_0(p)$. It is thus a coefficient module. The space $V_n$ inherits a left action by duality. We will also regard $V_n$ as a \emph{right} $\GL_2(\Qp)$-module in the usual way by setting
\[
(\omega\cdot \sigma)(P(X))=\omega(P(X)\cdot \sigma^{-1})
\]
for $\sigma \in \GL_2(\Qp)$. In this way, $V_n$ becomes a coefficient module and, in fact, $V_n$ is pure of weight $n$: if $\sigma=\smtx abcd$, then the action of $\sigma$ on $P_n$ is:
\[
P(X)\cdot \sigma=\det\sigma^{-n}(cX+d)^nP\left(\frac{aX+b}{cX+d}\right).
\]
We thus obtain the following right action on $V_n$:
\[
(\omega\cdot\sigma)(P(X))=\omega\left((-cX+a)^nP\left(\frac{dX-b}{-cX+a}\right)\right).
\]
Taking $a = d = \lambda$ and $b = c = 0$ shows that $V_n$ is pure of weight $n$.
\end{ex}

\begin{ex}
\label{ex:overconvergent}
Let $\bA$ denote the Tate ring in a single variable over $\Qp$:
\[
  \bA = \left\{\sum_{n \geq 0} a_n X^n \in \pseries{\Qp}{X} ~\Big{|}~ \abs{a_n} \to 0 \text{ as } n \to \infty\right\}.
\]
For each integer $n \geq 0$ let $\bA_n$ denote the ring $\bA$ endowed with the following left action of $\Sigma_0(p)$: for $\sigma = \twomat abcd$ in $\Sigma_0(p)$ and $f(X) \in \bA_n$, set
\[
  (\sigma \cdot f)(X) = (a - cX)^nf\left(\frac{-b+dX}{a-cX}\right).
\]
This action is well-defined since $c$ is divisible by $p$ and $a$ is a $p$-adic unit. However, it does not extend to an action of the full group $\GL_2(\Qp)$. Note that if one defines a left $\Sigma_0(p)$-structure on $P_n$ by setting $\sigma \cdot P = P\cdot \sigma^{-1}$, then the left actions of $\Sigma_0(p)$ on $P_n$ and $\bA_n$ are compatible with the inclusion of $P_n$ inside $\bA_n$.

Let $\bD_n$ denote the continuous dual of $\bA_n$, endowed with a right action of $\Sigma_0(p)$ by duality. This is a coefficient module, and in fact is pure of weight $n$. We refer to the spaces $\bD_n$ as spaces of \demph{rigid analytic distributions}. Note that the modules $\bD_n$ and $\bD_m$ are isomorphic as abstract $\QQ_p$-vector spaces, but not as representations of $\Sigma_0(p)$.
\end{ex}

\begin{rmk}
A third class of coefficient module was introduced by Pollack-Stevens in~\cite{pollackoverconvergent}. They are called \demph{finite approximation modules}, and are used to prove the theorems that we use to compute with rigid modular forms. If one is willing to accept those theorems, then it is not necessary to introduce these modules at all. We thus omit a precise description of these important objects.
\end{rmk}

\subsection{$p$-adic automorphic forms and harmonic cocycles}
\label{ssec:autforms}
Write $\Gamma_0(p\ZZ_p)=\GL_2(\QQ_p)\cap \Sigma_0(p)$.
\begin{dfn}
  Let $M$ be a coefficient module. A \demph{$p$-adic automorphic form for $\Gamma$} on $\GL_2(\QQ_p)$ with values
in $M$ is a left $\Gamma$-invariant and right $\Gamma_0(p\ZZ_p)$-equivariant map
$\varphi \colon \GL_2(\QQ_p) \to M$. The space of $M$-valued automorphic forms for $\Gamma$ on $\GL_2(\QQ_p)$ is denoted $\bA(\Gamma, M)$.
\end{dfn}

Fix a non-negative integer $n$ and let $V_n$ be as in Example~\ref{ex:Vncoeffmodule}.
\begin{dfn}
  The \emph{space of $p$-adic automorphic forms of weight $n$} for the group $\Gamma$ is $\bA_{n}(\Gamma)=\bA(\Gamma,V_{n})$.
\end{dfn}
  Given $\varphi \in \bA_{n}(\Gamma)$, we define an associated $V_n$-valued function $c_\varphi$ on the edges of $\cT$ as follows: suppose that $e\in E(\cT)$ is an even edge\footnote{Recall that an ordered edge of the Bruhat-Tits tree is said to be \emph{even} if its origin is an even distance from the privileged vertex $v_0$ corresponding to the lattice $\Zp^2$.}. Write $e=g \cdot e_0$ for some $g\in\GL_2(\QQ_p)$ and define $c_\varphi(e)=\varphi(g) \cdot g^{-1}$.

This is well-defined since the stabilizer of $e_0$ is the group $\Qp^\times\cdot\Gamma_0(p\ZZ_p)$, and $\varphi$ is invariant under $\Qp^\times$ and right equivariant under the action of $\Gamma_0(p\ZZ_p)$. If $e$ is an odd edge, then its opposite $\ebar$ is even. We extend $c_\varphi$ to the odd edges by setting $c_\varphi(e) = -c_\varphi(\ebar)$.

Conversely, given $c\in C_h(\Gamma,V_n)$, define $\varphi_c\in \bA_{n}(\Gamma)$ as follows: for $g\in \GL_2(\QQ_p)$, set $\varphi_c(g)= g^{-1}\cdot c(g \cdot e_0)$.

The assignments $\varphi\mapsto c_\varphi$ and $c\mapsto \varphi_c$ are well-defined and idenitfy $C_h(\Gamma,V_n)$ with the $p$-new subspace of $\bA_n(\Gamma)$. They thus yield an identification $\bA_{n}(\Gamma)^{p-\textrm{new}} \cong S_{n+2}(\Gamma)$ that is canonical up to a choice of parity for the edges. A change in parity results in the negative of the isomorphism above.\footnote{In the published version of this article we erroneously claimed that $C_h(\Gamma,V_n)$ is isomorhpic with the entire space $\bA_n(\Gamma)$ of $p$-adic automorphic forms. Harmonicity of the cocycles corresponds to the associated automorphic form being a $\Up$-eigenvector. The authors wish to thank James Newton for making us aware of this mistake.}

Let $M$ be an arbitrary coefficient module, and suppose that it is possible to represent $M$ on a computer (at least up to a prescribed precision). If $\vphi \in \bA(\Gamma, M)$, then $\vphi$ is determined completely by its values on a set of representatives for the finite double quotient $\Qp^\times\cdot\Gamma \backslash \GL_2(\QQ_p) / \Gamma_0(p\ZZ_p)$.  It is thus possible to store $\vphi$ as a vector of elements of $M$, which allows to compute with the spaces $\bA(\Gamma,M)$.

Let $\bD_n$ be as defined in Example \ref{ex:overconvergent}. The space of \demph{rigid analytic automorphic forms} of weight $n$ is the space of $p$-adic automorphic forms $\bbA_{n}(\Gamma) = \bA(\Gamma,\bD_{n})$.

Formation of spaces of $p$-adic automorphic forms is functorial in the coefficient modules. The natural inclusion $P_{n} \injects \bA_{n}$ yields a surjection $\bD_{n} \surjects V_{n}$ by duality, and thus defines a \demph{specialization map}
\[
  \rho \colon \bbA_{n}(\Gamma) \to \bA_{n}(\Gamma).
\]

\subsection{The $\Up$ operator}
\label{ssec:heckeops}
Let $M$ denote a coefficient module, and assume furthermore that it is pure of weight $n$ for some $n\geq 0$. In this subsection we give the action of the $\Up$ operator on the spaces $\bA(\Gamma,M)$, and specialize them to $A_n(\Gamma)$. We also describe the corresponding action imposed on $C_h(\Gamma,V_n) \cong A_n(\Gamma)$ by transport of structure.

Let $\alpha_0$ be the matrix $\smtx p001$, and write the double coset $\Gamma_0(p)\alpha_0\Gamma_0(p)$ as a disjoint union of left cosets
\[
\Gamma_0(p)\alpha_0\Gamma_0(p)=\bigcup_{i=0}^{p-1} \alpha_i \Gamma_0(p),\quad \alpha_i=\mtx pi01\quad\text{for $i=0\ldots p-1$}.
\]
\begin{dfn}
  The Hecke operator $\Up$ acts on $\bA(\Gamma,M)$ as:
\[
(\Up\varphi)(g)= p^{\frac{n}{2}}\sum_{i=0}^{p-1} \varphi(g\cdot \alpha_i)\cdot \alpha_i^{-1}.
\]
\end{dfn}

The action of $\Up$ on $\bA_n(\Gamma)$ is given by the same formula. Concretely,
\[
(\Up\varphi)(g) (P(x))= p^{\frac{n}{2}}\sum_{i=0}^{p-1} \varphi(g\cdot \alpha_i)(P(x)\cdot \alpha_i)
\]
Finally, we translate this definition to the space $C_h(\Gamma,V_n)$. We obtain:
  \[
  (\Up c)(e)(P(x))=p^{\frac{n}{2}}\sum_{\substack{o(e')=t(e)\\e'\neq \ol e}} c(e').
  \]

\begin{rmk}
It is a simple exercise to unwind the definitions above to work out formulae for the action of $\Up$ on $\bbA_n(\Gamma)$. Let $\Phi$ be an element of $\bbA_n(\Gamma)$. Fix a representative $b_j$, and for each $a\in\{0,\ldots p-1\}$, write
\[
b_j\smtx pa01 =\alpha \gamma b_{r(a)} \sigma_a,\quad \alpha\in\Qp^\times,\gamma\in\Gamma,\sigma\in\Gamma_0(p).
\]
Write also:
\[
\sigma_a\cdot (a+px)^i = \sum_{t\geq 0}\alpha_t x^t.
\]
With this notation, one finds:
\[
  (\Up\Phi)(b_j)(x^i) =  p^{-n/2}\sum_{a=0}^{p-1} \sum_{t \geq 0}\alpha_t \Phi(b_{r(a)})(x^t).
\]
This can be implemented on a computer since the values of $\Phi(b_{r(a)})$ are known.
\end{rmk}
\section{Applications}
\label{sec:applications}
In this final section we describe applications of all of the preceding objects and algorithms.

\subsection{Evaluation of rigid modular forms}
\label{ssec:computevalues}

Let $f \in S_{n+2}(\Gamma)$ and let $\varphi_f$ (resp. $\mu_f$) denote the corresponding $p$-adic automorphic form (resp. distribution). We store $\mu_f$ as a sequence of values in $V_{n}$ indexed by the edges of $\Gamma \backslash \cT$. From this data, one can evaluate $f$ at points $z \in \uhp_p$ via the Poisson kernel
\[
  f(z) = \int_{\PP^1(\Qp)} \frac{d\mu_f(t)}{z-t}
\]
using a process of polynomial approximation and Riemann integration. However, this method is hopelessly slow in practice: the number of balls in an exhaustion of $\PP^1(\Qp)$ grows exponentially as the radius of each of the balls shrinks, but the number of digits of precision grows only linearly with the radius. An alternate approach is thus necessary in order to compute values of rigid modular forms. Thankfully this was worked out in M. Greenberg's thesis~\cite{greenberg-thesis}, building on the work of Darmon-Pollack~\cite{darmonpollack}  and Pollack-Stevens~\cite{pollackoverconvergent}.

To explain how the integral above can be computed efficiently, first suppose that one had access to the following data: for every ball $B=\gamma U_{e_0}$ in $\PP^1(\QQ_p)$, the \demph{moments}
\[
  \int_B (\gamma^{-1}\cdot t)^id\mu_f(t)
\]
are given. With this at hand, one could break $\PP^1(\Qp)$ up into balls such that $1/(z-t)$ admits an analytic expression on each. Then one could use these locally analytic representations and use the corresponding moments to evaluate the requisite integral. In practice the number of balls arising in such a locally analytic representation is quite small. For example, if one wishes to evaluate $f$ at a point in the standard affinoid, then one can break $\PP^1(\Qp)$ up into the balls $a + p\Zp$ for $a = 0,\ldots, p-1$ and also $\{t ~|~ \abs{t} \geq p\}$. This skirts the exponential tesselation issue encountered above, but leaves us with the problem of computing the moments of $\mu_f$. Thankfully this problem has an elegant solution. We begin the explanation with the following key result.

\begin{thm}[Stevens]
\label{thm:uniquelift}
The restriction of the specialization map
\[
  \rho \colon \bbA_{n}(\Gamma)^{\leq {n}/2} \to \bA_{n}(\Gamma).
\]
is surjective, where the left-hand side is the subspace of $\bbA_n(\Gamma)$ on which $\Up$ acts with eigenvalue $\lambda$ such that $v_p(\lambda)\leq \frac{n}{2}$. In particular, given $\varphi\in\bA_n(\Gamma)$ there is a unique element $\Phi\in\bbA_n(\Gamma)$ such that $\rho(\Phi)=\varphi$ and such that $U_p\Phi=p^{\frac{n}{2}}\Phi$.
\end{thm}
\begin{proof}
  The original proof can be found in~\cite[Theorem~7.1]{stevens1994rigid}. For a constructive proof, the reader may consult~\cite[Theorem~5.12]{pollackoverconvergent}.
\end{proof}
\begin{thm}
  Let $n\geq 0$ be an even integer. Let $f \in S_{n+2}(\Gamma)$ denote a rigid analytic modular form, and let $\mu_f$ denote the corresponding $V_{n}$-valued distribution. Let $\Phi_f$ be the unique lift of $\varphi_f$ such that $\Up \Phi_f = p^{\frac{n}{2}}\Phi_f$. Then one has
\[
  \Phi_f(g)(t^i) = \omega(\mu_f,g,i) := \int_{g \cdot \ZZ_p} (g^{-1}\cdot t)^id\mu_f(t),\quad g\in\GL_2(\QQ_p).
\]
\end{thm}
\begin{proof}
Define $\Psi\in \bbA_n(\Gamma)$ as $\Psi(g)(\phi)=\int_{g \cdot\ZZ_p} \phi(g^{-1}\cdot t)d\mu_f(t)$. Note that $\rho(\Psi)=\varphi_f$, since for $i=0,\ldots,n$ we have
\begin{align*}
\Psi(g)(t^i)&= \int_{g \Zp}(g^{-1}\cdot t)^id\mu_f(t)= \mu_f(g\cdot \Zp)(t^i\cdot g^{-1})\\
&= g^{-1}\cdot\left(\mu_f(g\cdot \Zp)\right)(t^i) = (g^{-1}\cdot\mu_f(\Zp))(t^i).
\end{align*}
Also, using that $\bigcup_{i=0}^{p-1}\alpha_i\ZZ_p = \ZZ_p$ we can check that $\Up \Psi = p^{\frac{n}{2}} \Psi$ as follows.
\begin{align*}
(\Up \Psi)(g)(\phi(t))&=p^{\frac{n}{2}}\sum_{i=0}^{p-1}\Psi(g\alpha_i)(\phi(t)\cdot g\alpha_i)=p^{\frac{n}{2}}\sum_{i=0}^{p-1}\int_{g\alpha_i\ZZ_p}\phi(t) g\alpha_i (g\alpha_i)^{-1}d\mu_f(t)\\
&=p^{\frac{n}{2}}\sum_{i=0}^{p-1}\int_{g\alpha_i\ZZ_p}\phi(t)d\mu_f(t)=p^{\frac{n}{2}}\int_{g\ZZ_p}\phi(t)d\mu_f(t).
\end{align*}
Therefore the overconvergent automorphic form $\Psi$ must be $\Phi_f$, by the uniqueness statement of Theorem~\ref{thm:uniquelift}
\end{proof}

  Recall that, with notation as in the theorem above, $\Phi_f(\gamma)$ is an element in the dual of the Tate ring over $\Qp$. We represent such objects as a list of values in $\Qp$, where the $i$th entry of the list encodes the value of $\Phi_f(\gamma)$ applied to the rigid function $X^i$. Thus, the way that we represent the lift $\Phi_f$ on a computer encodes precisely all of the moments of the measure $\mu_f$ associated to $f$. We now explain how $\Phi_f$ may be computed up to an arbitrary $p$-adic precision by using an iterative procedure. The idea is to first lift $f$ to any form in $\bbA_k(\Gamma)$ specializing to $\varphi_f$, and then to iteratively apply $\widetilde \Up = p^{-n/2}\Up$ to this initial lift. The following result says that the resulting sequence converges linearly to $\Phi_f$.
\begin{thm}[{\cite[Proposition 5]{greenberg-thesis}}]
  Suppose that $\Phi \in \bbA_k(\Gamma)$ specializes to $\varphi_f$ and that $\Phi(\gamma)(X^i)=\omega(\mu_f,\gamma,i)$ for $i\leq i_0$. Then $\widetilde \Up\Phi$ also specializes to $f$, and $\Phi(\gamma)(X^i)=\omega(\mu_f,\gamma,i)$ for all $i\leq i_0+1$.
\end{thm}
Note that for $\Phi$ to specialize to $\varphi_f$ means exactly that the values of $\Phi(\gamma)$ on $X^i$ for $i\leq n$ agree with the first $n$ moments of $\mu_f$. Thus, the previous theorem shows that applying $\Up$ iteratively to an arbitrary lift of $\varphi_f$ will yield the moments of $\mu_f$ successively.

In practice one cannot compute all of the moments of $\mu_f$, due to limited time and memory resources. It is thus useful to know a priori how many moments are necessary for the computation of values of the rigid modular form $f$ up to a prescribed precision.
\begin{prop}[{\cite[Section 7]{greenberg2006heegner}}]
\label{p:numdigits}
  Let $z \in \uhp_p$, and assume for simplicity that $z$ belongs to the affinoid $\cA_0$. In order to compute the value $f(z)$ using the moments of $\mu_f$ to a $p$-adic precision of $N$ digits, one must precompute the first $N'$ moments of $\mu_f$ to precision $N''$, where
\[
N' = \max\{n\colon \ord_p(p^n/n)<N\},\quad N'' = M +\left\lfloor \frac{\log N'}{\log p}\right\rfloor
\]
\end{prop}

\begin{rmk} We remark that this method allows to integrate other locally analytic functions against measures $\mu_f$ using moments, as long as one can compute a locally analytic decomposition for the integrand.
\end{rmk}

In the remainder of this subsection we explain how the above can be implemented in practice on a computer. To begin, note that the set of double cosets $\Qp^\times\cdot\Gamma \backslash \GL_2(\Qp)/\Gamma_0(p\ZZ_p)$ is computed by Algorithm~\ref{alg:fdomain}. It allows us to write a decomposition
\begin{equation}
\label{eq:gl2decomp}
  \GL_2(\Qp) = \coprod_{j = 1}^g \Qp^\times\cdot\Gamma\cdot b_j \cdot \Gamma_0(p\ZZ_p)
\end{equation}
for some collection of matrices $b_j$ representing the ordered edges of $\Gamma\backslash \cT$. Note that we may also regard these matrices as representing ordered edges of $\cT$, that is, one of the many edges of $\cT$  which projects to the corresponding edge in the quotient graph. Let $U_j$ denote the open ball in $\PP^1(\Qp)$ consisting of all the ends which pass through the edge $b_j$ of $\cT$. This is precisely $b_j\cdot U_{e_0}$. Let $\mu_f$ denote the distribution attached to a rigid modular form $f$. Since it is invariant under the action of $\Gamma$, it is determined by the finitely many values $\mu_f(U_j)$ for $j = 1,\ldots, g$. We thus store $f$ as a sequence of values $\mu_f(U_j)$ in $V_{n}$. Elements of $V_{n}$ can be stored, up to a finite $p$-adic precision, by using the basis for $V_{n}$ dual to the basis of polynomials $1,X, \ldots, X^{n}$ for the symmetric representation $P_{n}$.

Even if one is content with storing values up to a finite $p$-adic precision, rigid automorphic forms are described by an infinite amount of data. Thus, when computing with rigid automorphic forms, one must have applications in mind at the outset. For example, if one is interested in computing values of rigid modular forms as above, then Proposition \ref{p:numdigits} shows how many moments are needed in order to perform the computation to a desired accuracy. With this application in mind, a rigid automorphic form $\Phi$ can be stored as a list of values $\Phi(b_j)$, where each entry is a list of the values it takes on $X^i$, with $i$ varying in some range $N$. With this description, the projection map
\[
  \rho \colon \bbA_{n+2}(\Gamma) \to \bA_{n+2}(\Gamma)
\]
simply forgets all but the first $n$ values of each $\Phi(b_j)$.

Suppose that $\varphi \in \bA_{n+2}(\Gamma)$ satisfies $\Up\varphi = p^{n/2}\varphi$. We seek to find $\Phi\in \bbA_{n+2}(\Gamma)$ such that $\rho(\Phi)=\varphi$. For each $j$ let
\[
  S_j = \Gamma_0(p\ZZ_p) \cap b_j^{-1}\Gamma b_j
\]
where $b_j$ is defined as in \ref{eq:gl2decomp}. Note that $S_j$ is contained in the global units $R^\times$, which is a finite group since $B$ is definite, and hence $S_j$ is finite. Furthermore, $b_jS_jb_j^{-1}$ is the stabilizer in $\Gamma$ of the ordered edge of $\cT$ represented by $b_j$, and so by the methods of Subsection \ref{ssec:equivprob}, the set $S_j$ is computable for each $j$. We may use these sets to average certain values of the distribution $\varphi$ to get an initial lift $\Phi_0 \in \bbA_{n+2}(\Gamma)$. For each $j$ set
\[
  \Phi_0(b_j) = \frac{1}{\# S_j}\sum_{v\in S_j} \varphi(b_j)\cdot v 
\]

The first $n$ values of the initial lift $\Phi_0$ are given by the corresponding residues of $f$, or equivalently, by the moments of the associated distribution $\mu_f$. Next one must repeatedly apply the operator $p^{-n/2} \Up$ to $\Phi_0$ in order to force the higher values of the lift to agree with the moments of $\mu_f$.
\begin{rmk}
  The presence of the factor $p^{-n/2}$ can cause one to lose precision when repeatedly applying $p^{-n/2}\Up$. This occurs also when one computes with overconvergent modular symbols -- see \cite[Remark~8.4]{pollackoverconvergent}. To compute the values $p^{-n/2}\Up\Phi_0(\gamma)(t^j)$, one must treat two cases separately: if $j \geq n/2$, then the division causes no problems as one knows the relevant moment $\Phi_0(\gamma)(t^j)$ to sufficient precision. One may thus simply apply the formulae defining $\Up$. If $n/2 > j$ then one uses the fact that $p^{-n/2}\Up\Phi_0(\gamma)(t^j)$ is the value of the specialisation $p^{-n/2}\rho_{n+2}(\Up\Phi_0)(\gamma)$ evaluated on the polynomial $t^j$. Since the values of the specialisation are assumed to be known to the target precision, one can ensure that no precision is lost.
\end{rmk}

\subsection{Equations for Shimura curves}
\label{ssec:equations}
In this final section we describe a method that uses the above techniques to produce, in favorable cases, convincing conjectural equations for the canonical embedding of a $p$-adically uniformizable Shimura curve $X/\QQ$. We thus assume that $X$ is not hyperelliptic (hence $g\geq 3$), for two reasons: first, so that the canonical embedding yields a smooth projective model for $X$ and second, since equations for hyperelliptic Shimura curves have been described in~\cite{2010arXiv1004.3675M}. In fact, we will restrict to genus $4$ below. The case of genus $3$ is simpler, since the canonical embedding yields a hypersurface in $\PP^2$. We believe that one can push our method, part of which derives from~\cite{MR1297415}, to treat genera greater than $4$.

Start by calculating a basis $S=(c_1,\ldots,c_g)$ of eigenvectors for $C_h(\Gamma)$. This can be done by diagonalizing the action of $T_l$ on $M_2(X)$ for enough distinct primes $l$. Let $(f_1,\ldots,f_g)$ be the corresponding rigid analytic modular forms on $\cH_p$. These are defined over a number field $H$, and we have described how to evaluate these forms on points $z\in\cH_p$. Choose a prime $\frakp$ of $H$ above $p$, and let $K\subset \CC_p$ be a finite extension of $\QQ_p$ containing both the completion of $H$ at $\frakp$ and the quadratic unramified extension of $\QQ_p$.

The canonical embedding $X_\Gamma(K) \to\PP^{g-1}(K)$ is the map with the $f_i$ as coordinates:
\[
P\mapsto (f_1(P)\colon\cdots\colon f_g(P)).
\]
Since $X_\Gamma$ is not hyperelliptic this is indeed a closed embedding, for which we wish to find the defining ideal $I\subset H[x_1,\ldots,x_g]$.

We now specialize to the case $g = 4$. By~\cite[Example~5.2.2]{hartshorne1977ag}, $I=(F,G)$, where $F$ and $G$ are homogeneous polynomials of degree $2$ and $3$, respectively. Express them in the form $F(x_1,\ldots,x_4)=\sum_{I \in E_2} a_I x^I$, $,\quad G(x_1,\ldots,x_4)=\sum_{I \in E_3} a_I x^I$, where $E_2$ and $E_3$ denote the sets of exponent vectors $I$ of $4$ non-negative integers adding up to $2$ and $3$, respectively.

Denote by $N$ and $M$ the number of monomials in $4$ variables of degree $2$ and $3$, respectively. Let $E=E_2\cup E_3$ and set $V=\PP^N\times \PP^M$, with coordinates $(a_I)_{I\in E}$. The torus $T= \CC_p^4$ acts on $V$ by
\[
\lambda\cdot a_I:= \left(\prod_{i=1}^4 \lambda_i^{e_i}\right)a_I\quad \lambda = (\lambda_1,\ldots, \lambda_4) \in T,~ I = (e_1,\ldots, e_4) \in E.
\]
%

The Hilbert finiteness theorem ensures that the algebra of polynomial invariants on $V$ is finitely-generated. We can strengthen this result as follows:
\begin{prop}
  The field of weight-$0$ invariant functions on $V$ is finitely-generated over the base field.
\end{prop}
\begin{proof}
Consider the variety $V\times\bA^g$. Adjoin to its ring of regular functions as many auxiliary variables as the dimension of the torus. Endow the $i$th variable with weight $-1$ with respect to the $i$th coordinate of the torus, and $0$ else. Now just note that for every homogeneous rational function of weight $w\geq 0$ of the original ring one can construct an invariant polynomial of this new ring, by multiplying the original function by an appropriately chosen monomial. The Hilbert finiteness theorem implies finite generation of the augmented ring. Since the field of $T$-invariant rational functions on $V\times\bbA^g$ of weight $0$ is generated by monomials, the invariant rational functions (now again of the original ring) are generated by monomials.
\end{proof}

Integral linear algebra allows one to find generators for the algebra of invariant monomials of degree $0$ of the variety cut out by the polynomials $F$ and $G$. Since the Shimura curve has a model defined over $\QQ$, the resulting invariants will necessarily be rational, although computing them via the canonical embedding as described above only yields a $p$-adic approximation. Thus, one must try to recognize the resulting $p$-adic invariants as rational ones, say by using the LLL algorithm. One is frequently successful in obtaining convincing conjectural rational invariants.

The dimension of the algebra of invariants is in general smaller than the dimension of the variety $V$, and therefore such a procedure will not yield a unique rational equation for $V$. One needs to use arithmetic to pin down the last parameters, and the following result is useful for this last step.

\begin{prop}
\label{prop:alfixedpoints}
  The points on the Shimura curve $X_0(N,1)$ fixed by the Atkin-Lehner involutions $w_d$ for $d\mid N$ are the CM points attached to the orders:
  \begin{align*}
    &\ZZ[\sqrt{-1}]\text{ and }\ZZ[\sqrt{-2}] &\text{if }d&=2\\
    &\ZZ[\sqrt{-d}]\text{ and }\ZZ\left[\dfrac{1+\sqrt{-d}}{2}\right] &\text{if }d&\equiv 3\pmod {4}\\
    &\ZZ[\sqrt{-d}] &\text{if }d&\equiv 1,2 \pmod{4}\\
  \end{align*}
\end{prop}
\begin{proof}
  See~\cite{clarkthesis}
\end{proof}

We conclude with an illustration of this method in the case of the genus $4$ curve $X = X_0(53\cdot 2,1)$ uniformised at the prime $p=53$. An application of Algorithm \ref{alg:fdomain} yields a quotient graph with two vertices and $5$ edges. We thus hope to find an integral model for $X$ whose fiber at $p = 53$ is a semistable curve equal to two genus $0$ curves that intersect in $5$ regular double points. 

Begin by computing a basis of rigid analytic modular forms which are eigenvectors for the Hecke algebra. In this example the eigenvalues are rational integers, so that in the previous notation $H=\QQ$. By evaluating monomials in these funtions of degree $2$ and $3$ at random points in the quadratic unramified extension $K$ of $\QQ_p$, one obtains a quadratic relation $\widetilde F$ and a cubic relation $\widetilde G$ with coefficients in $\QQ_p$. In this example the only monomials that appear in $\widetilde{F}$ are
\[
  x_0^2,~ x_0x_2,~ x_1^2,~ x_2^2 \text{ and } x_3^2,
\]
while $\widetilde{G}$ is expressed in terms of
\[
  x_0^3,~ x_0^2x_2,~ x_0x_1^2,~ x_1^2x_2, \text{ and } x_2^3.
\]
Note that we cannot say for certain that other monomials do not appear in $\widetilde{F}$ and $\widetilde{G}$, as these relations were obtained using inexact $p$-adic methods. In this computation we worked with a $53$-adic accuracy of $100$ digits.

The pair $(\widetilde F,\widetilde G)$ describes a point in the product of projective spaces $V$ that lies in the same $T$-orbit as the desired rational pair $(F,G)$. Thus, the value of any rational invariants on this orbit must be rational. Hence, by computing the invariants using the $p$-adic point $(\widetilde F,\widetilde G)$, one obtains a $p$-adic approximation to a rational invariant. If one writes
\begin{align*}
F&=a_0x_0^{2}+a_1x_0 x_2+a_2x_1^2+a_3x_2^{2}+a_4x_3^{2}\\
G&=b_0x_0^{3}+b_1x_0^{2} x_2+b_2x_0 x_1^{2}+b_3x_1^2x_2+b_4x_2^{3},
\end{align*}
then one computes the following relations:
\begin{align*}
\frac{a_0a_2^2}{b_2^2}&\equiv \frac{25}{9} \pmod{53^{100}} &\frac{a_2^3b_4}{b_3^3}&\equiv \frac{16}{135}\pmod{53^{100}}\\
\frac{a_1a_2^2}{b_2b_3}&\equiv\frac{10}{3}\pmod{53^{100}} &\frac{a_2^3b_0}{b_2^3}&\equiv \frac{125}{27}\pmod{53^{100}}\\
\frac{a_3a_2^2}{b_3^2}&\equiv\frac{-17}{18}\pmod{53^{100}} &\frac{a_2^3b_1}{b_2^2 b_3}&\equiv \frac{55}{9}\pmod{53^{100}}.
\end{align*}

Assuming that the above congruences are in fact equalities allows us to eliminate the variables $a_0,a_1,a_3$ and $b_0,b_1,b_3$ to get
\begin{align*}
F&=\frac{25}{9}C^2x_0^2 + \frac{10}{3}CDx_0x_2 +B^3x_1^2-\frac{17}{18}D^2x_2^2 + AB^2x_3^2,\\
G&=\frac{125}{27}C^3x_0^3 + \frac{55}{9}C^2Dx_0^2x_2  + B^3Cx_0x_1^2 + B^3Dx_1^2x_2 + \frac{16}{135}D^3x_2^3.
\end{align*}
The substitutions
\begin{align*}
  x_0&\mapsto \frac{x_0}{5C},&
  x_1&\mapsto \frac{x_1}{3B},&
  x_2&\mapsto \frac{x_2}{D},&
  x_3&\mapsto \frac{x_3}{3B},
\end{align*}
yield the simpler equations
\begin{align*}
  F&=2 x_0^{2} + 12 x_0 x_2 + 2Bx_1^2-17x_2^2+18Ax_3^2.\\
G&=5 x_0^{3} + 33 x_0^{2} x_2 + 3 B x_0 x_1^{2} + 15 B x_1^2 x_2+16x_2^3. 
\end{align*}
Note that, at the cost of rescaling $x_1$ and $x_3$ by rational factors, we may assume that the unknowns $A$ and $B$ are squarefree integers.

Each of the basis elements $x_0,x_1,x_2,x_3$ corresponds to an elliptic curve of conductor $106$. In terms of Cremona labels, they are respectively \emph{106a}, \emph{106b}, \emph{106c} and \emph{106d}. The Atkin-Lehner involutions $w_2$ and $w_{53}$ act by:
\begin{align*}
  w_2(x_0\colon x_1\colon x_2\colon x_3)&= (x_0\colon -x_1\colon x_2\colon -x_3)\\
w_{53}(x_0\colon x_1 \colon x_2\colon x_3)&=(x_0\colon -x_1\colon x_2\colon x_3).
\end{align*}

Since the point $(0:1:0:0)$ is not on the curve defined by $F$ and $G$, the divisor of $x_1$ is supported on the fixed points of $w_{53}$. Similarly, the divisor of $x_3$ is supported on the fixed points of $w_{2\cdot 53}$. By Proposition~\ref{prop:alfixedpoints} these fixed points are CM points attached to certain orders of imaginary quadratic fields. The CM theory of Shimura curves implies that the coordinates of these points generate the ring class field of the corresponding order.

The coordinates of the fixed points of $w_{53}$ are $(\alpha \colon0\colon 1\colon \beta)$, where $\alpha$ is a root of the polynomial $5t^3+33t^3+16$ and $\beta$ satisfies
\[
2\alpha^2+12\alpha-17+2 A \beta^2=0.
\]
Let $M$ be the Galois closure of the field $\QQ(\sqrt{-53},\alpha,\beta)$. By Proposition~\ref{prop:alfixedpoints} this should be the Hilbert class field of $\QQ(\sqrt{-53})$, and therefore it is required that $A$ is (up to squares) either $-2\cdot 3$ or $2\cdot 3\cdot 53$. If $A$ is divisible by any primes other than $2$ or $53$ then the ramification in $M$ is too large, and for the primes $2$ and $53$ one may exhaust all possibilities on a computer. However, up to a $p$-adic square $A$ is explicitely computed to be equal to
\[
47 + 53 + 2\cdot 53^2 + 5\cdot 53^3 + 10\cdot 53^4 +\cdots.
\]
Therefore $A=-6$. A similar computation with $w_{2\cdot 53}$ yields $B=-3$. We thus obtain the (conjectural) equations
\begin{align}
    F&=2 x_0^{2} + 12 x_0 x_2 -6x_1^2- 17 x_2^{2} - 108 x_3^{2}, \\\nonumber
G&=5 x_0^{3} + 33 x_0^{2} x_2 -9x_0x_1^2- 45 x_1^2 x_2+16x_2^3.
\end{align}
This model does not have good reduction at $2$, $3$ or $53$. Outside of these primes the model has good reduction. At $p = 2$ the fiber is not reduced and hence not semistable. No rescaling of the coordinates yields a model with semistable reduction at $p =2$ or good reduction at $p=3$. At $p = 53$ the reduction is semistable.

We used Magma to count points over finite fields in order to find the Euler factors of the $L$-series of the curve described by $F$ and $G$ for primes up to $5000$. The $L$-series of the Shimura curve $X$ is a product of $L$-series of rational elliptic curve. We verified that the Euler factors computed from our equations agree with the expected values obtained from these elliptic curves. While one could try to use the method of Faltings-Serre as in~\cite{serreletter} to prove the accuracy of the proposed equations for $X_0(53\cdot 2,1)$, we feel that the preceding computations alone provide a compelling showcase of the use of the $p$-adic algorithms described in this paper.

\bibliographystyle{amsalpha}
\bibliography{refs}
\end{document}